\documentclass[12pt]{article}
\usepackage[utf8]{inputenc}
\usepackage[T2A]{fontenc}
\usepackage{cmap}
\usepackage{indentfirst}

\usepackage[english]{babel}
\usepackage{amsmath,amssymb,amsthm,graphicx}
\usepackage{bbm}
\usepackage[usenames]{color}
\usepackage{colortbl}
\usepackage{caption}
\usepackage{subcaption}
\usepackage[unicode, pdftex]{hyperref}
\usepackage[left=2.5cm,right=2.2cm,top=2.5cm,bottom=3.5cm]{geometry}
\usepackage{wasysym}

\newcommand\R{\mathbb R}
\newcommand\Z{\mathbb Z}

\newtheorem{theorem}{Theorem}[section]

\newtheorem{proposition}{Proposition}[section]
\newtheorem{lemma}{Lemma}[section]
\newtheorem{conj}{Conjecture}[section]
\newtheorem{definition}{Definition}[section]
\newtheorem{notation}{Notation}[section]
\newtheorem{Corollary}{Corollary}[section]

\begin{document}

\title{The exact bound for the reverse isodiametric problem in 3-space
}
\author{Arkadiy Aliev}
\date{}
\maketitle

    \begin{abstract}
        Let $K$ be a convex body in $\mathbb{R}^{3}$. We denote the volume of $K$ by $Vol(K)$ and the diameter of $K$ by $Diam(K).$  In this paper we prove that there exists a linear bijection $T:\mathbb{R}^{3}\to \mathbb{R}^{3}$ such that $Vol(TK)\geq \frac{\sqrt{2}}{12}Diam(TK)^3$ with equality if $K$ is a simplex, which was conjectured by Endre Makai Jr.  
        
        As a corollary, we prove that any non-separable lattice of translates in $\mathbb{R}^{3}$ has density of at least $\frac{1}{12}$, which is a dual analog of Minkowski's fundamental theorem.  Also we prove that  $Vol(K)\geq \frac{1}{12}\omega(K)^3$, where $K\subset \mathbb{R}^{3}$ is a convex body and  $\omega(K)$ is the lattice width of $K$. In addition, there exists a  three-dimensional simplex $\Delta\subset \mathbb{R}^3$ such that $Vol(\Delta) = \frac{1}{12}\omega(\Delta)^3.$ 
    \end{abstract}

\section*{Organisation of the paper}

        The paper is organised in the following way. In the first section, we give basic definitions, review  the literature and the best-known results on the reverse isodiametric problem and on the non-separable lattices. In the sections 2-7 we give the proof of the reverse isodiamteric problem in 3-space.   
    \tableofcontents

\begin{section}{Introduction}

    We call a convex compact set $K\subset\mathbb{R}^{n}$ with non-empty interior a convex body. Denote the volume of $K$ by $Vol(K)$ and the diameter of $K$ by $Diam(K).$ For any $\lambda \in \mathbb{R}$ we define $\lambda K := \{\lambda x \text{ }\vert \text{ } x\in K\}$. $K$ is centrally symmetric if $K = -K$. The difference body of $K$ is defined as $K-K:=\{x-y \text{ }\vert\text{ } x,y\in K\}$.
    
    The isodiamteric inequality is a classic result in convex geometry, which states  that the Euclidean ball in $\mathbb{R}^n$ has the maximum volume among all convex bodies of the same diameter, but it is impossible to give a non-trivial lower bound for the volume of a convex body in terms of its diameter. However, if one first minimizes the diameter using a volume-preserving linear mapping, then a non-trivial bound exists. 

    \begin{definition}
        Let $K \subset \mathbb{R}^{n}$ be a convex body. Its isodiametric quotient is defined as 
        $$
            idq(K) = \frac{Vol(K)}{Diam(K)^n}.
        $$
    \end{definition}    

    Endre Makai Jr.  posed the conjecture, which states that there exists a linear image of a convex body $K$ such that its isodiametric quotient is at least as large as  that of a regular simplex.
    
    \begin{conj}[Makai Jr. \cite{makai1978}]
        For every convex body $K\subset \mathbb{R}^n$ there exist a linear bijection $T$ such that 
        $$
            idq(TK)\geq \frac{\sqrt{n+1}}{n! 2^{\frac{n}{2}}}.
        $$
    \end{conj}
    
    The two dimensional case of the conjecture was found by Behrend in 1937 \cite{behrend}. Cases of dimension three and more still remain open and the best results in them were obtained by Merino B.G. and  Schymura M. in \cite{shimura}. 

    \begin{proposition}[Merino-Schymura, \cite{shimura}] For a convex body $K \subset \mathbb{R}^{n}$ there exists a linear bijection $T$ such that

    $$
        idq(TK)\geq \frac{\left(\frac{n+1}{2}\right)^{\frac{n}{2}}}{\binom{n(n+1)/2}{n}^{1/2}n!} \sim \sqrt{e}\left(\frac{2\pi}{n}\right)^{\frac{1}{4}}\frac{\sqrt{n+1}}{n!e^{\frac{n}{2}}}.
    $$
        
    \end{proposition}

    In the three dimensional case they proved that for a convex body $K\subset \mathbb{R}^3$ there exists a linear bijection $T$ such that $idq(TK)\geq \frac{\sqrt{10}}{30}$, which is very close to $\frac{\sqrt{2}}{12}$. In this paper, we prove the Makai's conjecture in dimension three.

    \begin{theorem}
    \label{main-th}
        Let $K$ be a convex body in $\R^{3}$. Then there exists a linear bijection $T:\mathbb{R}^{3}\to \mathbb{R}^{3}$ such that $$
            Vol(TK)\geq \frac{\sqrt{2}}{12}Diam(TK)^3.
        $$      
    \end{theorem}

    Note that if a convex body $K \subset \mathbb{R}^{n}$ is centrally symmetric, then it is known that there exists a linear bijection $T$ such that $idq(TK)\geq \frac{1}{n!}$  with equality if and only if $TK$ is a regular crosspolytope (see \cite{barthe}).

    The reverse isodiametric inequality has applications in the geometry of numbers. $\Lambda\subset \R^{n}$ is a lattice if $\Lambda = A\mathbb{Z}^n$ for some $A\in GL(n,\mathbb{R})$. We also define $d(\Lambda) := |det A|.$ A lattice of translates of $K$ is defined as $\Lambda + K = \{x+y \text{ } \vert\text{ } x\in \Lambda,y\in K\}$ and its density is defined as $$D(\Lambda, K) := \frac{Vol(K)}{d(\Lambda)}.$$

    We say that a         lattice of translates of $K$ is non-separable if each affine $(n-1)$-subspace in $\R^{n}$ meets $x+K$ for some $x\in\Lambda$. Non-separable lattices were introduced by  L. Fejes T\'oth in \cite{makai1974}, see also \cite{makai1978, makai2016}. 
    \begin{definition} For a convex body $K\subset \mathbb{R}^{n}$ we define 
        $$d_{n,n-1}(K) := \min\{d(\Lambda,K)\text{ }\vert\text{ } \Lambda+K \text{ is non-separable} \}.$$ 
    \end{definition}

    The estimates of $d_{n,n-1}$ are related to the Mahler's conjecture and have been well studied in the two-dimensional case (see \cite{ makai2016, AAA23}). Namely, for a convex body $K\subset \mathbb{R}^2$ we have
    
        $$
            \frac{3}{8}\leq d_{2,1}(K)\leq\frac{\pi \sqrt{3}}{8}, $$ 
    and if $K \subset \mathbb{R}^{2}$ is centrally symmetric, then 
    $$
        d_{2,1}(K)\geq \frac{1}{2}.
    $$
    
    The proof of the Mahler's conjecture in dimension three gives us the estimate  $d_{3,2}(K)\geq \frac{1}{6}$ in the \textit{centrally symmetric} case, but the general case remains open. 
    As a consequence of the reverse isodiametric inequality, we obtain the following result. 
    
    \begin{Corollary}
    \label{d32}
         Let $K$ be a convex body in $\R^{3}$. Then 
         $$
            d_{3,2}(K)\geq \frac{1}{12}.
         $$
    In addition, if $K$ is a simplex, then $d_{3,2}(K) = \frac{1}{12}.$
    \end{Corollary}

    This corollary can be reformulated in terms of the covering minimum. Let $\mathcal{A}_{i}(\mathbb{R}^{n})$ be the family of $i$-dimensional affine subspaces of $\mathbb{R}^{n}$. Then for a convex body $K\subset \mathbb{R}^{n}$ and a lattice $\Lambda \subset \mathbb{R}^{n}$ the $i$-th covering minimum is defined as 
    $$
        \mu_{i}(K, \Lambda) = \inf\{\mu > 0 \text{ }\vert\text{ } (\mu K + \Lambda) \cap L \neq \emptyset \text{ for all $L \in \mathcal{A}_{n-i}(\mathbb{R}^{n})$ }\}.
    $$
    Then we get the following result, which can be seen as a dual analog of Minkowski’s fundamental theorem (see \cite{bernardo}). 
    \begin{Corollary}
    \label{cor1.2}
        For a convex body $K \subset \mathbb{R}^{3}$ and a lattice $\Lambda \subset \mathbb{R}^{3}$ we have 
        $$
            \mu_{1}(K, \Lambda)^{3}Vol(K) \geq \frac{1}{12}\det \Lambda,
        $$
        and equality can hold if $K$ is a simplex.
    \end{Corollary}

    In order to explain the duality we introduce Minkowski’s successive minimum. For a centrally symmetric convex body $K\subset \mathbb{R}^{n}$ and a lattice $\Lambda \subset \mathbb{R}^{n}$ the $i$-th successive minimum is defined as 

    $$
        \lambda_{i}(K,\Lambda) = \min \{\lambda > 0 \text{ }\vert \text{ } dim(\lambda K \cap \Lambda) \geq i\} \text{ for $i=1,\ldots, n$}.
    $$
    For general convex body $K$ this definition is extended by setting $\lambda_{i}(K,\Lambda) := \lambda_{i}(\frac{1}{2}(K-K), \Lambda).$ In these notations 
 Minkowski’s fundamental theorem states that 

$$
    \lambda_{1}(K,\Lambda)^{n}Vol(K)\leq 2^{n}\det \Lambda.
$$

The last step to understand the duality is the identity of Kannan and Lovász \cite{lovasz}. Namely, for a centrally symmetric convex body $K\subset \mathbb{R}^{n}$ one has 
$$
    \lambda_{1}(K,\Lambda)\mu_{1}(K^{\circ}, \Lambda^{\circ}) = \frac{1}{2},
$$
where $K^{\circ} = \{x\in\mathbb{R}^{n} \text{ }\vert\text{ }\langle x,y\rangle\leq 1 \text{ for all }y\in K\}$ is the polar body of $K$ and $\Lambda^{\circ} = \{x\in \mathbb{R}^{n} \text{ }\vert\text{ }\langle x,y\rangle \in \mathbb{Z} \text{ for all }y\in \Lambda\}$ is the polar lattice of $\Lambda$. Thus one can see the polar duality between Corollary \ref{cor1.2} and Minkowski’s theorem.

    The next corollary is an estimate of the volume of a  convex body in terms of its lattice width, which is of independent interest due to its connection to the Newton 
polyhedrons of surfaces with $m$-singularity. 

 \begin{definition} Let $K$ be a convex body in $\R^{n}$. The 
        lattice width $\omega_{K}:\Z^{n}\setminus{0}\to \R$ is defined as 
        $$
        \omega_{K}(u)=\max_{x,y\in K}\langle u, x-y\rangle;
        $$
        $$
            \omega(K) = \min_{u\in \mathbb{Z}^{n}\setminus 0} \omega_{K}(u).
        $$
    \end{definition}

There are estimates of the area of the Newton polygons in two dimensions (see \cite{nk1,nk2}), but in three dimensions the exact estimates remained unknown. We prove the following result.  
    \begin{Corollary}
    \label{lattice-width}
        Let $K$ be a convex body in $\R^{3}$. Then 
         $$
            Vol(K)\geq \frac{1}{12}\omega(K)^3.
         $$
    In addition, if $K = conv\left\{(0,0,0),  \left(1,\frac{1}{2},\frac{1}{2}\right),\left(\frac{1}{2},1,\frac{1}{2}\right),\left(\frac{1}{2},\frac{1}{2},1\right)\right\}$, then $Vol(K) = \frac{1}{12}\omega(K)^3.$
    \end{Corollary}

    It is no coincidence that the coefficient $\frac{1}{12}$ is the same in both corollaries. The following relation between non-separable lattices and lattice widths was obtained in \cite{makai1974}.

    \begin{proposition}
    [T\' oth-Makai, \cite{makai1974}]
    \label{nonsep}  Let $K$ be a convex body in $\R^{n}$.
        $K+\Z^{n}$ is non-separable if and only if $\omega(K)\geq 1$.  And therefore $K+\Z^{n}$ is non-separable if and only if $\frac{1}{2}(K-K)+\Z^{n}$ is non-separable.  
    \end{proposition}

\end{section}

\begin{section}{The scheme of the proof}

  In this section we outline the proof of Theorem~\ref{main-th}. In the proof, we will need to estimate the volume of the convex hull of the segments. First, recall the following well-known assertion.

\begin{proposition} 
    \label{app}
    
    Let $K\subset \mathbb{R}^{d}$ be a convex body and $\underline{z}^{i},\overline{z}^{i}\in K$ for $i =1,\ldots,d.$ And let $y_{i} = \overline{z}^{i}-\underline{z}^{i}.$ Then it holds that

    $$
        Vol(K)\geq \frac{|det(y_{1},\ldots,y_{d})|}{d!}.
    $$
        
    \end{proposition}

    \begin{proof}[Proof of Proposition \ref{app}]
        Let us prove this statement by induction. If $d=1$ then there is nothing to prove. Assume that $d>1$ and let $\Pi := y_{d}^{\perp}$ and $P$ be the orthogonal projection onto $\Pi$. Then $$|det(y_{1},\ldots,y_{d})| = |y_{d}|\cdot|det(Py_{1},\ldots,Py_{d-1})|$$ and by (1.3) in \cite{zhang} we have 
        $$
            Vol(K)\geq \frac{1}{d}|y_{d}|Vol(PK) \geq \frac{1}{d}|y_{d}|\frac{|det(Py_{1},\ldots,Py_{d-1})|}{(d-1)!} = \frac{|\det(y_{1},\ldots,y_{d})|}{d!}.
        $$
    
    \end{proof}

        Further, we define the required linear mapping $T$. Let $\mathcal{E}$ be the minimal volume ellipsoid containing $K-K$  and $B$ be the unit ball in $\mathbb{R}^{3}.$ $T$ is a linear bijection such that $T\mathcal{E} = B.$ Then by the famous John's theorem \cite {ball} we get the existence of contact points $u_{1},\ldots,u_{6}\in (TK-TK)\cap \partial B$ and \textit{non-negative numbers} $\lambda_{1},\ldots,\lambda_{6}, $ such that 

         $$
            \sum_{i=1}^{6}\lambda_{i}u_{i}u_{i}^{T} = Id_{3},
        $$
        where we denoted the identity matrix $3\times 3$ by $Id_{3}$. Then $TK$ contains shifted copies of $u_{1},\ldots,u_{6}$ and $Diam(TK)=1$. Our goal now is to show that $Vol(K)\geq \frac{\sqrt{2}}{12}$. Therefore by Proposition~\ref{app} it suffices to prove that $\max_{i,j,k}|det(u_{i},u_{j},u_{k})|\geq \frac{1}{\sqrt{2}}.$ 

        From the trace comparison we get $\sum_{i=1}^{6} \lambda_{i} = 3.$ Let $\lambda_{6} = \max_{i}\lambda_{i}$ and  $a_{ij}:=det(u_{i},u_{j},u_{6}).$ Then by Lemma~\ref{det-like} and Lemma~\ref{id-cor} we have the following equalities. 

        $$
            a_{13}a_{24} - a_{14}a_{23} = a_{12}a_{34};
        $$
        $$
            a_{13}a_{25} - a_{15}a_{23} = a_{12}a_{35};
        $$
        $$
            a_{14}a_{25} - a_{15}a_{24} = a_{12}a_{45};
        $$
        $$
            a_{14}a_{35} - a_{15}a_{34} = a_{13}a_{45};
        $$
        $$
            a_{24}a_{35} - a_{25}a_{34} = a_{23}a_{45};
        $$

        $$
            \sum_{1\leq i<j\leq 5} \lambda_{i}\lambda_{j}a_{ij}^2= 1.
        $$
        Under this restrictions it suffices to prove that $\max_{i,j}|a_{ij}|\geq \frac{1}{\sqrt{2}}.$ It turns out to be more convenient to solve the following problem:

        \textbf{The optimization problem.} Assume that $\lambda_{1},\ldots, \lambda_{6} \geq 0$ are fixed,  $\sum_{i=1}^{6}\lambda_{i} =3 $ and $\lambda_{6} = \max_{i}\lambda_{i}$. Let $\{a_{ij}\}_{1\leq i<j\leq 5}$ be a set of $10$ numbers, such that the list of equalities from Lemma~\ref{det-like} holds and $\max_{1\leq i<j\leq 5} |a_{ij}| \leq 1$. Find the maximum value of $$\sum_{1\leq i<j\leq 5}\lambda_{i}\lambda_{j}a_{ij}^2.$$ 
        If the maximum value does not exceed $2$, then we win. Our goal is to prove this. 
        
    \begin{definition}
        We say that the set $\{a_{ij}\}_{1\leq i<j\leq 5}$ is admissible, if  $\max_{1\leq i<j\leq 5} |a_{ij}| \leq 1$ and the following holds.
        $$
            a_{13}a_{24} - a_{14}a_{23} = a_{12}a_{34};
        $$
        $$
            a_{13}a_{25} - a_{15}a_{23} = a_{12}a_{35};
        $$
        $$
            a_{14}a_{25} - a_{15}a_{24} = a_{12}a_{45};
        $$
        $$
            a_{14}a_{35} - a_{15}a_{34} = a_{13}a_{45};
        $$
        $$
            a_{24}a_{35} - a_{25}a_{34} = a_{23}a_{45}.
        $$
    
    \end{definition}

\begin{definition} Assume that $\lambda_{1},\ldots, \lambda_{6}$ are fixed. 
    We say that the admissible set  $\{a_{ij}\}_{1\leq i<j\leq 5}$ is maximal if it maximizes the value of  $\sum_{1\leq i<j\leq 5}\lambda_{i}\lambda_{j}a_{ij}^2$ among all admissible sets.
\end{definition}

The proof of the estimate is divided into four parts. The first part, which is Section \ref{s5}, is related to the case where there is at least one zero among $\lambda_{1},\ldots,\lambda_{6}.$ The second part, which is Section \ref{s6}, is related to finding special conditions for maximal sets. If all $\lambda_{i}$ are positive, then there are only two types of maximal sets: "peculiar sets"~\ref{def-bad} and those with at least one zero. The third part, which is Section \ref{s7}, is related to the case where the set $\{a_{ij}\}$ is "peculiar". And the last fourth part, which is Section \ref{s8} is related to the case where the set $\{a_{ij}\}$ contains zero. Each of the cases is based on estimates of sums of lambda products, which are proved in the Lemmas~\ref{lemma-2}, \ref{lemma3.4}, \ref{lemma3.5}, \ref{lemma3.6}. After solving all this cases, Theorem~\ref{main-th} will be proved. 

Summing up, the remaining steps are the following:

\begin{enumerate}
    \item Solve the case $\lambda_{1}=0.$ It is done in Proposition~\ref{lambda-zero-case}.
    \item If $\lambda_{i}>0$ for all $i$, prove that there are only two types of maximal sets: "peculiar sets" and those with at least one zero. It is done in Proposition~\ref{max-form}. 
    \item Solve the case of peculiar sets. It is done in Proposition~\ref{bad-case}.
    \item Solve the case  $a_{12}=0$. It is done in Proposition~\ref{zero-case}.
    
\end{enumerate}

    During proofs, the words \textbf{we may ignore some  expression} will often be used. What we mean by that words is the following. Assume, for example, that we have a function $$(x,y,z)\mapsto ax^2 + by^2 + cz^2$$ with the domain $[-1;1]^{3}\cap \{x+y+z=0\}$ and $a,b,c > 0$. We want to estimate it as $a+b+c - \min(a,b,c).$ We know the obvious estimate $a+b+c$, but we want to be more precise. This function is convex, therefore it suffices to consider only 
    
    $$(x,y,z)\in \{(1,-1,0),(-1,1,0),(1,0,-1),(-1,0,1),(0,1,-1), (0,-1,1)\}.$$  
    In each of this cases either $a$ or $b$, or $c$ comes with zero coefficient. And this means that we may ignore either $a$ or $b$, or $c$.

    \begin{proof}[Proof of Corollaries \ref{d32}, \ref{cor1.2} and \ref{lattice-width}]
        This reasoning is well known, we present it for completeness.

        Let $K + \Lambda$ is non-separable and let $T:\mathbb{R}^3\to \mathbb{R}^3.$ be a linear bijection such that $$Vol(TK)\geq \frac{\sqrt{2}}{12}Diam(TK)^3.$$
        Then $TK + T\Lambda$ is non-separable and therefore by Proposition \ref{nonsep} $\frac{1}{2}(TK-TK) + T\Lambda$ is non-separable. Since $\frac{1}{2}(TK-TK)\subset B_{\frac{1}{2}Diam(TK)}$, $B_{\frac{1}{2}Diam(TK)} + T\Lambda$ is non-separable.
        
        Therefore its density is at least $\frac{\pi}{6\sqrt{2}}$ by Theorem 4 of \cite{makai1978}. Then we have the following chain of inequalities.
        $$
            \frac{Vol(K)}{det \Lambda} = \frac{Vol(TK)}{det T\Lambda} \geq \frac{\sqrt{2}}{12}\frac{Diam(TK)^3}{det (T\Lambda)} = \frac{8\sqrt{2}}{12\frac{4}{3}\pi}\frac{Vol(B_{\frac{1}{2}Diam(TK)})}{det (T\Lambda)} \geq \frac{8\sqrt{2}}{12\frac{4}{3}\pi}\frac{\pi}{6\sqrt{2}} = \frac{1}{12}.
        $$
        Therefore $d_{32}(K)\geq \frac{1}{12}.$ And, since $K+\mathbb{Z}^3$ is non-separable if and only if $\omega(K)\geq 1$, we have  $Vol(K)\geq \frac{1}{12}\omega(K)^3.$  If $K$ is a simplex, then  $d_{3,2}(K) = \frac{1}{12}$ by Proposition 3.1 in \cite{makai2016}.
        
    \end{proof}

\begin{section}{Preliminary calculations}

    Let $u_{1},\ldots,u_{6}$ be unit vectors in $\mathbb{R}^{3}$ and for some non-negative 
    $\lambda_{1},\ldots,\lambda_{6}$ we have 
    $$
        \sum_{i=1}^{6}\lambda_{i}u_{i}u_{i}^{T} = Id_{3}.
    $$    
Then from the trace comparison we get $\sum_{i=1}^{6} \lambda_{i} = 3.$ Note that $u_{i}u_{i}^{T}\in \mathbb{R}^{3\times 3}.$
    
    \begin{notation}
        We denote $det(u_{i},u_{j},u_{6})$ by $a_{ij}$. Note that $a_{ij} = -a_{ji}.$
    \end{notation}

    We will need some properties of $a_{ij}$. Firstly, we have determinant-like equations:
    
    \begin{lemma}
    \label{det-like}
        
    The following list of equalities holds.
        $$
            a_{13}a_{24} - a_{14}a_{23} = a_{12}a_{34};
        $$
        $$
            a_{13}a_{25} - a_{15}a_{23} = a_{12}a_{35};
        $$
        $$
            a_{14}a_{25} - a_{15}a_{24} = a_{12}a_{45};
        $$
        $$
            a_{14}a_{35} - a_{15}a_{34} = a_{13}a_{45};
        $$
        $$
            a_{24}a_{35} - a_{25}a_{34} = a_{23}a_{45}.
        $$
    \end{lemma}

    \begin{proof}
       We can assume that $u_{6} = (0,0,1)^{T}$. Let us rewrite $u_{i}$ as $ u_{i} = (x_{i},y_{i},*).$
        Then we have
        $$
            a_{12} = x_{1}y_{2}-x_{2}y_{1}; 
        $$
        $$
            a_{13} = x_{1}y_{3}-x_{3}y_{1};
        $$
        $$
            a_{14} = x_{1}y_{4}-x_{4}y_{1};
        $$
        $$
            a_{23} = x_{2}y_{3}-x_{3}y_{2};
        $$
         $$
            a_{24} = x_{2}y_{4}-x_{4}y_{2};
        $$
        $$
            a_{34} = x_{3}y_{4}-x_{4}y_{3}.
        $$
        Direct calculation completes the proof.
    \end{proof}

    Let $u,v$ be two vectors  in $ \mathbb{R}^{3}$. We denote their \textbf{cross product} as $[u;v]$. 
     Then from the decomposition of the identity matrix, we obtain the following:

    \begin{lemma}
    \label{id-cor}
        $$
            \sum_{1\leq i<j\leq 5} \lambda_{i}\lambda_{j}a_{ij}^2= 1.
        $$
    \end{lemma}
    
    \begin{proof}
        From \cite{g-ivanov} we have the following equality:

        $$
            \sum_{1\leq i<j\leq 6}\lambda_{i}\lambda_{j}[u_{i}; u_{j}][u_{i}; u_{j}]^{T} = Id_{3},
        $$
        where $[u_{i}; u_{j}]$ is the cross product of $u_{i},u_{j}$.
        Then since $|u_{6}|=1$ we get

        $$
            \sum_{1\leq i<j\leq 6}\lambda_{i}\lambda_{j}[u_{i}; u_{j}]\langle[u_{i}; u_{j}], u_{6} \rangle = u_{6}
        $$

        $$
         \sum_{1\leq i<j\leq 5} \lambda_{i}\lambda_{j}a_{ij}^2 = \sum_{1\leq i<j\leq 5}\lambda_{i}\lambda_{j}\langle[u_{i}; u_{j}], u_{6} \rangle ^{2}= 1.
        $$
        
    \end{proof}
    
    Now let us move on to optimizing the sum of lambda products, which we will often use in the main proof.

    \begin{lemma}
    \label{lemma-2}
        Let $\lambda_{6} = \max_{i}\lambda_{i}$ and $k,l,m,n$ be four different  numbers from $\{1,2,3,4,5\}$.Then we have 
        $$
            \sum_{1\leq i<j\leq 5}\lambda_{i}\lambda_{j} - \lambda_{k}\lambda_{l} - \lambda_{m}\lambda_{n}\leq 2.
        $$
    \end{lemma}
    
    \begin{proof}
        We can assume that $\lambda_{1}\leq \lambda_{2}\leq \ldots \leq \lambda_{5}$ and since under such  assumptions $\lambda_{1}\lambda_{4} + \lambda_{2}\lambda_{3}$ is the smallest sum, we can also put $n=1,m=4,k=2,l=3$. 
         $$
            (\lambda_{1} + \lambda_{4})(\lambda_{2}+\lambda_{3} + \lambda_{5}) + \lambda_{5}(\lambda_{2} + \lambda_{3}) \leq (\lambda_{1} + \lambda_{4})\left(\lambda_{2}+\lambda_{3} + \frac{\lambda_{5}+\lambda_{6}}{2}\right) + \frac{\lambda_{5}+\lambda_{6}}{2}(\lambda_{2} + \lambda_{3})  
    $$

    Let $a = \lambda_{1} + \lambda_{4}$, $b = \lambda_{2} + \lambda_{3}$. Then $a+2b\leq 3$ and $2a+b\leq 3$. Our goal is to estimate the following expression
    $$
        a\left(b+\frac{3}{2} - \frac{a+b}{2}\right) + \left(\frac{3}{2} - \frac{a+b}{2}\right)b = \frac{3}{2}(a+b) - \frac{1}{2}(a^{2}+b^{2}) =: F(a,b).
    $$

    It suffices to consider the case $0\leq a\leq b$ and $a+2b \leq 3$, in other words, the case $(a,b) \in conv\{(0,0),(0,3/2), (1,1)\}$. First of all, the only point with $\frac{\partial}{\partial a}F(a,b) = \frac{\partial}{\partial b}F(a,b) = 0$ is $(\frac{3}{2},\frac{3}{2}) \not\in conv\{(0,0),(0,3/2), (1,1)\}$. Therefore it suffices to prove the inequality only for $(a,b) \in \partial conv\{(0,0),(0,3/2), (1,1)\}.$ 
    
     If $a+2b=3$, then we have 
    $$
        \frac{3}{2}(a+b) - \frac{1}{2}(a^{2}+b^{2}) = \frac{9}{2}b - \frac{5}{2}b^2\leq 2 \text{ for $b\geq 1$.}
    $$
    
    If $a=0$, then we have 

     $$
        \frac{3}{2}(a+b) - \frac{1}{2}(a^{2}+b^{2}) = \frac{3}{2}b - \frac{1}{2}b^2 \leq \frac{9}{8}<2.
    $$
    
    If $a=b$, then we have 

     $$
        \frac{3}{2}(a+b) - \frac{1}{2}(a^{2}+b^{2}) = 3a - a^2 \leq 2 \text{ for $a\leq 1$.}
    $$
    \end{proof}
    
    \begin{lemma}
    \label{lemma3.4}
        Let $\lambda_{6} = \max_{i}\lambda_{i}$ and $k,l,n$ be three different  numbers from $\{1,2,3,4,5\}$. Then we have 
        $$
            \sum_{1\leq i<j\leq 5}\lambda_{i}\lambda_{j} - \lambda_{k}\lambda_{l} - \lambda_{l}\lambda_{n} - \lambda_{k}\lambda_{n}\leq \frac{9}{5} < 2.
        $$
    \end{lemma}
    
    \begin{proof}
        We can assume that $\lambda_{1}\leq \lambda_{2}\leq \ldots \leq \lambda_{5}$ and that $k=1, l=2, n=3$. Our goal is to find the upper bound for 
        $$          (\lambda_{1}+\lambda_{2}+\lambda_{3})(\lambda_{4}+\lambda_{5}) + \lambda_{4}\lambda_{5}.
        $$
        Let $a := \frac{1}{3}(\lambda_{1} + \lambda_{2} + \lambda_{3})$ and $b = \frac{1}{3}(\lambda_{4} + \lambda_{5} + \lambda_{6})$. Then we have $a + b = 1 ,$ $a\leq b$ and 
        $$
            (\lambda_{1}+\lambda_{2}+\lambda_{3})(\lambda_{4}+\lambda_{5}) + \lambda_{4}\lambda_{5} \leq 6ab + b^2 = 6b - 5b^2 \leq \frac{9}{5}.
        $$

    \end{proof}
    As a consequence, we obtain the following lemma. 
     \begin{lemma}
    \label{lemma3.5}
         Let $\lambda_{6} = \max_{i}\lambda_{i}$ and $\lambda_{1}=0$. $k\neq l \in \{2,3,4,5\}$ Then we have 
         $$
            \sum_{1\leq i<j\leq 5}\lambda_{i}\lambda_{j}  - \lambda_{k}\lambda_{l} \leq \frac{9}{5} 
 < 2.
         $$
    \end{lemma}

    \begin{proof}
          $$
            \sum_{1\leq i<j\leq 5}\lambda_{i}\lambda_{j}  - \lambda_{k}\lambda_{l} = \sum_{1\leq i<j\leq 5}\lambda_{i}\lambda_{j}  - \lambda_{k}\lambda_{l} - \lambda_{k}\lambda_{1} - \lambda_{l} \lambda_{1}  \leq \frac{9}{5} 
         $$
    \end{proof}

    \begin{lemma}
    \label{lemma3.6}
         Let $\lambda_{6} = \max_{i}\lambda_{i}$. Then we have
         
        $$
           \frac{3}{5}(\lambda_{1}\lambda_{5} + \lambda_{1}\lambda_{4} +\lambda_{2}\lambda_{3} + \lambda_{2}\lambda_{5} + \lambda_{3}\lambda_{4} ) + \lambda_{1}\lambda_{2} + \lambda_{1}\lambda_{3} + \lambda_{2}\lambda_{4} +\lambda_{3}\lambda_{5} + \lambda_{4}\lambda_{5} \leq 2
        $$
    \end{lemma}

    \begin{proof}
        Applying Lemma~\ref{lemma-2} for $$(k,l,m,n) = (1,4,2,3),
        (1,4,2,5),(1,5,2,3), (1,5,3,4), (2,5,3,4)$$ gives us the following:

        $$
            \sum_{1\leq i<j\leq 5}\lambda_{i}\lambda_{j} - \lambda_{1}\lambda_{4} - \lambda_{2}\lambda_{3}\leq 2;
        $$
        $$
            \sum_{1\leq i<j\leq 5}\lambda_{i}\lambda_{j} - \lambda_{1}\lambda_{4} - \lambda_{2}\lambda_{5}\leq 2;
        $$
        $$
            \sum_{1\leq i<j\leq 5}\lambda_{i}\lambda_{j} - \lambda_{1}\lambda_{5} - \lambda_{2}\lambda_{3}\leq 2;
        $$
        $$
            \sum_{1\leq i<j\leq 5}\lambda_{i}\lambda_{j} - \lambda_{1}\lambda_{5} - \lambda_{3}\lambda_{4}\leq 2;
        $$
        $$
            \sum_{1\leq i<j\leq 5}\lambda_{i}\lambda_{j} - \lambda_{2}\lambda_{5} - \lambda_{3}\lambda_{4}\leq 2.
        $$

        The desired estimate is obtained by summation.
        
    \end{proof}

\end{section}

\begin{section}{The case  \texorpdfstring{$\lambda_{1} = 0$}{Lg}}
\label{s5}

\begin{proposition}
\label{lambda-zero-case}
    Let $  \lambda_{1}=0$ and $\lambda_{6} = \max_{i}\lambda_{i}$. Then for an admissible set $\{a_{ij}\}_{1\leq i<j\leq 5}$ we have
    $$
        \sum_{1\leq i < j\leq 5}\lambda_{i}\lambda_{j} a_{ij}^{2} \leq \frac{9}{5}<2.
    $$
\end{proposition}

\begin{proof}
    We can assume that $\lambda_{2}\leq \lambda_{3}\leq \lambda_{4} \leq \lambda_{5}.$  Then we can evaluate pairs of terms as follows:
    
    $$
        \lambda_{2}\lambda_{4}a_{24}^{2} + \lambda_{3}\lambda_{5}a_{35}^{2} \leq \lambda_{2}\lambda_{4}(a_{24}a_{35})^{2} + \lambda_{3}\lambda_{5};
    $$

       $$
        \lambda_{2}\lambda_{3}a_{23}^{2} + \lambda_{4}\lambda_{5}a_{34}^{2} \leq \lambda_{2}\lambda_{3}(a_{23}a_{45})^{2} + \lambda_{4}\lambda_{5};
    $$

    $$
        \lambda_{2}\lambda_{5}a_{25}^{2} + \lambda_{3}\lambda_{4}a_{34}^{2} \leq \min(\lambda_{2}\lambda_{5},\lambda_{3}\lambda_{4} )(a_{25}a_{34})^{2} + \max(\lambda_{2}\lambda_{5},\lambda_{3}\lambda_{4} ).
    $$
    Let $u=a_{24}a_{35}, v = a_{23}a_{45}; w = a_{25}a_{34}.$ $u-v=w.$  Consider the following cases:
    
    \textbf{Case 1: $uv\leq 0$.} 

    $$
        \lambda_{2}\lambda_{4}u^{2} + \lambda_{2}\lambda_{3}v^{2} \leq   \lambda_{2}\lambda_{4}(u^{2} + v^{2}) \leq \lambda_{2}\lambda_{4}(u - v)^{2} \leq \lambda_{2}\lambda_{4}.
    $$

    $$
     \min(\lambda_{2}\lambda_{5},\lambda_{3}\lambda_{4} )(a_{25}a_{34})^{2} + \max(\lambda_{2}\lambda_{5},\lambda_{3}\lambda_{4} ) \leq \lambda_{2}\lambda_{5} + \lambda_{3}\lambda_{4}
    $$

    Therefore we get
    
    $$
        \sum_{1\leq i < j\leq 5}\lambda_{i}\lambda_{j} a_{ij}^{2} \leq \lambda_{3}\lambda_{5} + \lambda_{4}\lambda_{5} + \lambda_{2}\lambda_{4} + \lambda_{2}\lambda_{5} + \lambda_{3}\lambda_{4}\leq \frac{9}{5},
    $$
    where the last inequality was obtained using Lemma~\ref{lemma3.5}. 
    
    \textbf{Case 2: $uv > 0$.}
    Without loss of generality, we can assume that 
    $u,v>0$.
    
    \textbf{Case 2.1 $w = u-v \leq 0$.}
    
    Then $ uw\leq 0$ and  

    $$
        \min(\lambda_{2}\lambda_{5},\lambda_{3}\lambda_{4} )w^{2} + \lambda_{2}\lambda_{4}u^{2}
        $$
        
        $$
        \leq \min(\lambda_{2}\lambda_{5},\lambda_{3}\lambda_{4} )(w^{2} + u^{2}) \leq \min(\lambda_{2}\lambda_{5},\lambda_{3}\lambda_{4} )(w - u)^{2} \leq \min(\lambda_{2}\lambda_{5},\lambda_{3}\lambda_{4} );
    $$

     Therefore we get
    
    $$
        \sum_{1\leq i < j\leq 5}\lambda_{i}\lambda_{j} a_{ij}^{2} \leq \lambda_{3}\lambda_{5} + \lambda_{4}\lambda_{5} + \lambda_{2}\lambda_{3} + \lambda_{2}\lambda_{5} + \lambda_{3}\lambda_{4}\leq \frac{9}{5},
    $$
    where the last inequality was obtained using Lemma~\ref{lemma3.5}.

    \textbf{Case  2.2 $w = u-v \geq 0$.} then $vw\geq 0.$
    $$
        \min(\lambda_{2}\lambda_{5},\lambda_{3}\lambda_{4} )w^{2} + \lambda_{2}\lambda_{3}v^{2}     
        $$
        $$\leq \min(\lambda_{2}\lambda_{5},\lambda_{3}\lambda_{4} )(w^{2} + v^{2}) \leq \min(\lambda_{2}\lambda_{5},\lambda_{3}\lambda_{4} )(w + v)^{2} \leq \min(\lambda_{2}\lambda_{5},\lambda_{3}\lambda_{4} )
    $$

     Therefore we get
    
    $$
        \sum_{1\leq i < j\leq 5}\lambda_{i}\lambda_{j} a_{ij}^{2} \leq \lambda_{3}\lambda_{5} + \lambda_{4}\lambda_{5} + \lambda_{2}\lambda_{4} + \lambda_{2}\lambda_{5} + \lambda_{3}\lambda_{4}\leq \frac{9}{5},
    $$
    where the last inequality was obtained using Lemma~\ref{lemma3.5}.

\end{proof}
    
\end{section}

\end{section}

\begin{section}{Maximal sets}
\label{s6}

    Our main goal in this section is to find the special properties of maximal sets.
    
    \begin{definition}
    \label{def-bad}
        We call the admissible set $\{a_{ij}\}_{1\leq i<j\leq 5}$  \textbf{ peculiar} if after some permutation of indices it holds that 
        $$
            |a_{12}|=|a_{13}|=|a_{24}|=|a_{35}|=|a_{45}|=1;
        $$ 
        
        $$
            |a_{15}| + |a_{14}|\geq 1;
        $$ 
        
        $$
            |a_{23}| = \frac{|a_{14}| + |a_{15}| - 1}{|a_{14}a_{15}|};
        $$
        $$
            |a_{25}| = \frac{1-|a_{15}|}{|a_{14}|};
        $$
        $$
            |a_{34}| = \frac{1-|a_{14}|}{|a_{15}|}.
        $$
    \end{definition}
    
    \begin{proposition}
        \label{max-form}
        Assume that $\lambda_{i}> 0$ for all $i\in\{1,2,3,4,5\}$ and that the set $A=\{a_{ij}\}_{1\leq i<j\leq 5}$ is maximal. Then $A$ is  peculiar or $A$ contains zeros. 
    \end{proposition}
    
    \begin{proof}
        Consider the maximal set $\{a_{ij}\}_{1\leq i < j\leq 5}$ and assume that it does not contain zero. Then we fix $a_{12}$ and define $u=(a_{13}, a_{14}, a_{15})^{T}, v=(a_{23},a_{24},a_{25})^{T}, w = (-a_{45},a_{35},-a_{34})^{T}$. Note that the following relation holds:
    
        $$
            w = \frac{1}{a_{12}}[u; v].
        $$
        Moreover, if we have a set $\{a_{ij}\}_{1\leq i < j\leq 5}$ without zeros, with this relation and all its elements lie in $[-1;1]$, then the set is admissible. 
    
         We define positive matriсes $A,B,C \in \mathbb{R}^{3\times 3}$ as follows:  
        $$
            A = diag(\lambda_{1}\lambda_{3}, \lambda_{1}\lambda_{4},\lambda_{1}\lambda_{5});
        $$ 
        
        $$
            B = diag(\lambda_{2}\lambda_{3}, \lambda_{2}\lambda_{4},\lambda_{2}\lambda_{5});
        $$
        
        $$ 
            C = diag(\lambda_{4}\lambda_{5}, \lambda_{3}\lambda_{5},\lambda_{3}\lambda_{4}).
        $$ 
        Let us rewrite $\sum_{1\leq i < j\leq 5}\lambda_{i}\lambda_{j} a_{ij}^{2}$ in terms of new variables $u,v,[u;v], a_{12}$.
        $$
            \sum_{1\leq i < j\leq 5}\lambda_{i}\lambda_{j} a_{ij}^{2} = \langle Au, u\rangle + \langle Bv, v\rangle + \frac{1}{a_{12}^2}\langle C ([u;v]), [u;v]\rangle + \lambda_{1}\lambda_{2}a_{12}^2.
        $$
    
        Since the set $\{a_{ij}\}$ is admissible, we know that $u,v,w \in [-1;1]^{3}.$ Our goal now is to show that $u,v \in \partial [-1;1]^{3}.$ 
    
        Consider $tu, \frac{1}{t}v$. Then our expression becomes
    
        $$
            \sum_{1\leq i < j\leq 5}\lambda_{i}\lambda_{j} a_{ij}^{2}(t) = \alpha t^2 + \beta \frac{1}{t^2} + \gamma,
        $$
        for some $\alpha, \beta > 0.$ Then since $\alpha t^2 + \beta \frac{1}{t^2} + \gamma$ is strictly convex, its maximum is reached only  at the boundary. Therefore $u \in \partial [-1;1]^{3}$ or $v \in \partial [-1;1]^{3}.$ Let us assume that $u\in\partial [-1;1]^{3}.$ Further, let us fix $u$ and $[u;v]$. Then all $v\in [-1;1]^3$ such that $[u;v]$ is fixed form a segment $\{(1-s)v_{0} + sv_{1}| t\in [0;1]\}$ inside the cube with $v_{0},v_{1}\in \partial[-1;1]^{3}.$ So let $v_{s} = (1-s)v_{0} + sv_{1}. $ Then our expression becomes
    
        $$
            \sum_{1\leq i < j\leq 5}\lambda_{i}\lambda_{j} a_{ij}^{2}(s) = \tilde{\alpha} s^2 + \tilde{\beta}s + \tilde{\gamma},
        $$
        where  $\tilde{\alpha} = \langle B(v_{1}-v_{0}), v_{1}-v_{0}\rangle \geq 0. $ If $\tilde{\alpha} = 0$, then $v_{t}=v_{0}=v_{1}\in\partial[-1;1]^{3}$. If $\tilde{\alpha}>0$,  then, since $\tilde{\alpha} s^2 + \tilde{\beta}s + \tilde{\gamma}$ is strictly convex, its maximum is reached only  at the boundary. Therefore $v \in \partial [-1;1]^{3}$.
    
        We showed that if the set $\{a_{ij}\}_{1\leq i < j \leq 1}$ is maximal and $a_{12}\neq 0$, then $$
            (a_{13}, a_{14}, a_{15})^{T}, (a_{23},a_{24},a_{25})^{T} \in \partial [-1;1]^{3}. 
        $$ 
        Similarly, if $i,k,l,m$ are four distinct indices from $\{1,2,3,4,5\}$ and the set $\{a_{ij}\}_{1\leq i < j \leq 1}$ is maximal, then $(a_{ik}, a_{il},a_{im})^{T} \in \partial [-1;1]^{3}$. Note that if $i>j$, then $a_{ij}:=-a_{ji}$.
    
        The remaining step is to use the relations from Lemma~\ref{det-like} and the condition $$(a_{ik}, a_{il},a_{im})^{T} \in \partial [-1;1]^{3}$$ to understand how does a maximal set look like.
    
        There are at least two ones among $|a_{12}|,|a_{13}|,|a_{14}|,|a_{15}|$. Assume without loss of generality that  $|a_{12}|=|a_{13}|=1.$ Then there is at least one $1$ among $|a_{23}|, |a_{24}|,|a_{25}|$. Since the case where $|a_{24}|=1$ and the case where $|a_{25}|=1$ are the same, there are only two cases: $|a_{23}|=1$ and $|a_{24}|=1.$
    
        \textbf{Case 23. ($|a_{12}|=|a_{13}|=|a_{23}|=1$).} In this case there is at least one $1$ among $|a_{42}|, |a_{43}|,|a_{45}|$. From symmetry, it suffices to consider only two cases: $|a_{42}|=1$ and $|a_{45}|=1.$
    
        \textbf{Case 23.42 ($|a_{12}|=|a_{13}|=|a_{23}|=1$, $|a_{42}|=1$)} The relation  $a_{13}a_{24} - a_{14}a_{23} = a_{12}a_{34}$  gives us  $|a_{14}| + |a_{34}|=1.$ Further, there is at least one $1$ among $|a_{41}|,|a_{43}|, |a_{45}|$. If $|a_{41}|=1$ or $|a_{43}|=1$, then from $|a_{14}| + |a_{34}|=1$ we can conclude that the remaining equals $0$. Therefore the remaining case is $|a_{45}|=1.$ Then there is at least one $1$ among $|a_{15}|$, $|a_{25}|$, $|a_{35}|$.
    
         \textbf{Case 23.42.15($|a_{12}|=|a_{13}|=|a_{23}|=1$, $|a_{42}|=1$, $|a_{45}|=|a_{15}|=1$)} The relations 
         
         $$
            a_{13}a_{24} - a_{14}a_{23} = a_{12}a_{34},
         $$
           $$
            a_{13}a_{25} - a_{15}a_{23} = a_{12}a_{35},
         $$
         $$
            a_{14}a_{35} - a_{15}a_{34} = a_{13}a_{45}
         $$
         give us $|a_{14}| + |a_{34}| = 1;$ $|a_{25}| + |a_{35}| = 1;$ $|a_{14}a_{35}| + |a_{15}a_{34}| = 1$. Assume that $a_{14}\neq 0$, then $|a_{35}|=1$ and therefore $a_{25}=0.$

         \textbf{Case 23.42.25($|a_{12}|=|a_{13}|=|a_{23}|=1$, $|a_{42}|=1$, $|a_{45}|=|a_{25}|=1$)} The relations 
         
         $$
            a_{13}a_{24} - a_{14}a_{23} = a_{12}a_{34},
         $$
           $$
            a_{13}a_{25} - a_{15}a_{23} = a_{12}a_{35},
         $$
         $$
            a_{14}a_{35} - a_{15}a_{34} = a_{13}a_{45}
         $$
         give us $|a_{14}| + |a_{34}| = 1;$ $|a_{15}| + |a_{35}| = 1;$ $|a_{14}a_{35}| + |a_{15}a_{34}| = 1$. Then $|a_{14}a_{15}| + |a_{34}a_{35}|=0$ and therefore there is at least one zero in the set.
         
         \textbf{Case 23.42.35($|a_{12}|=|a_{13}|=|a_{23}|=1$, $|a_{42}|=1$, $|a_{45}|=|a_{35}|=1$)} 
     The relations 
         
         $$
            a_{13}a_{24} - a_{14}a_{23} = a_{12}a_{34},
         $$
           $$
            a_{13}a_{25} - a_{15}a_{23} = a_{12}a_{35},
         $$
       $$
            a_{14}a_{25} - a_{15}a_{24} = a_{12}a_{45};
        $$
         give us $|a_{14}|+|a_{34}|=1; |a_{25}| + |a_{15}|=1; |a_{14}a_{25}| + |a_{15}|=1.$ Assume that $a_{25}\neq 0$, then $|a_{14}|=1$ and therefore $a_{34}=0.$ 
    
        \textbf{Case 23.45 ($|a_{12}|=|a_{13}|=|a_{23}|=1$, $|a_{45}|=1$)} Then there is at least one $1$ among $|a_{14}|,|a_{24}|,|a_{34}|$ and all these cases are similar to the Case $23.42$.

        \textbf{Case 24. ($|a_{12}|=|a_{13}|=|a_{24}|=1$).} Then there is at least one $1$ among $|a_{32}|,|a_{34}|,|a_{35}|$. The case where $|a_{32}|=1$ is done in the Case $23.42$, so there are only two left: $|a_{34}|=1$ and $|a_{35}|=1$.
    
        \textbf{Case 24.34 ($|a_{12}|=|a_{13}|=|a_{24}|=1$, $|a_{34}|=1$).} Then the relation 
        $$
            a_{13}a_{24} - a_{14}a_{23} = a_{12}a_{34}
        $$
        gives us $|a_{14}a_{23}|=0$.
    
        \textbf{Case 24.35 ($|a_{12}|=|a_{13}|=|a_{24}|=1$, $|a_{35}|=1$).}  Then there is at least one $1$ among $|a_{41}|,|a_{43}|,|a_{45}|$. The case where $|a_{12}|=|a_{24}| = |a_{14}| = 1$ is similar to the case where $|a_{12}|=|a_{13}| = |a_{23}|=1$, which is Case $23$. Therefore there are two cases left: $|a_{43}|=1$ and $|a_{45}|=1$.
    
         \textbf{Case 24.35.43 ($|a_{12}|=|a_{13}|=|a_{24}|=1$, $|a_{35}|=1$, $|a_{43}|=1$).} Then the relation 
        $$
            a_{13}a_{24} - a_{14}a_{23} = a_{12}a_{34}
        $$
        gives us $|a_{14}a_{23}|=0$.
    
          \textbf{Case 24.35.45 ($|a_{12}|=|a_{13}|=|a_{24}|=1$, $|a_{35}|=1$, $|a_{45}|=1$).} This is the case of peculiar sets. The relations give us the following system of equations: 
    $$
        1 = |a_{34}| + |a_{14}a_{23}|;
    $$
    
    $$
        1 = |a_{25}| + |a_{15}a_{23}|;
    $$
    
    $$
        1 = |a_{15}| + |a_{14}a_{25}|;
    $$
    
    $$
        1 = |a_{14}| + |a_{15}a_{34}|;
    $$
    
    $$
        1 = |a_{23}| + |a_{25}a_{34}|.
    $$
    
    If we solve it with respect to $|a_{14}|,|a_{15}|$, we get 
    
        $$
            |a_{15}| + |a_{14}|\geq 1;
        $$ 
        
        $$
            |a_{23}| = \frac{|a_{14}| + |a_{15}| - 1}{|a_{14}a_{15}|};
        $$
        $$
            |a_{25}| = \frac{1-|a_{15}|}{|a_{14}|};
        $$
        $$
            |a_{34}| = \frac{1-|a_{14}|}{|a_{15}|}.
        $$
        
    \end{proof}

\end{section}

\begin{section}{The case of peculiar sets}
\label{s7}
    In this case, we need two more technical lemmas. 
    \begin{lemma}
    \label{omega}
        Let 
        $$
            \Omega = \{(x,y) \in [0;1]^{2} \text{ }\vert \text{ } x\geq \frac{1}{2}; y\geq\frac{1}{2};xy\leq \frac{1}{2};2y-xy\leq 1; 2x-xy\leq 1;  \};
        $$
        $$
            g:(x,y)\mapsto \left(\frac{1-x}{1-xy}, 1-xy\right ).
        $$
        Then $g(\Omega)\subset \Omega.$
    \end{lemma}
    
    \begin{proof}
        
        We need to check five inequalities.
        $$
            2x-xy\leq 1 \Leftrightarrow 1 - xy \leq 2 -2x \Leftrightarrow \frac{1}{2} \leq \frac{1-x}{1-xy};
        $$
    
        $$
            xy \leq \frac{1}{2} \Leftrightarrow \frac{1}{2}\leq 1 - xy;
        $$
    
        $$
            x\geq \frac{1}{2} \Leftrightarrow 1 - x \leq \frac{1}{2} \Leftrightarrow \frac{1-x}{1-xy}(1-xy)\leq \frac{1}{2};
        $$
    
        $$
            y\geq \frac{1}{2} \Rightarrow x-2xy\leq 0 \Rightarrow 2(1-xy) - (1-x) \leq 1.
        $$
       The remaining one requires additional calculations:
        $$
            2\frac{1-x}{1-xy} - (1-x) \stackrel{?}{\leq} 1 \Leftrightarrow 
        $$
        $$
            2(1-x)- (1-x)(1-xy) \stackrel{?}{\leq} 1 -xy \Leftrightarrow
        $$
        $$
            2-2x- 1+x+xy-x^2y  \stackrel{?}{\leq} 1 -xy \Leftrightarrow
        $$
    
          $$
            -x + 2xy-x^2y  \stackrel{?}{\leq} 0 \Leftrightarrow x(-1 + 2y-xy)  
     \leq 0 
        $$
    \end{proof}
    
    \begin{lemma}
    \label{lemma7.2}
        Assume that  $(x,y)\in \Omega$. Then 
        $$
            \max \left(x^{2},y^{2},(1-xy)^{2}, \left(\frac{1-x}{1-xy}\right)^{2}, \left(\frac{1-y}{1-xy}\right)^{2}\right)\leq \frac{9}{16} < \frac{3}{5}.
        $$
    \end{lemma}
    
    \begin{proof}
        $1\leq 2x \leq 1 + xy \leq \frac{3}{2}$ and  $1\leq 2y \leq 1 + xy \leq \frac{3}{2}$.   From Lemma~\ref{omega} we have $  \left(\frac{1-x}{1-xy}, 1-xy\right ) \in \Omega, $ so the same holds for them. 
        
    \end{proof}

    \begin{lemma}
    \label{lemma7.3}
    Let $\lambda_{6} = \max_{i}\lambda_{i}$ and for $x,y \in [0;1]^{2}\setminus(1,1)$ we define $f(x,y)$ as follows:
        $$
            f(x,y) := \lambda_{1}\lambda_{5}\left(\frac{y-1}{xy-1}\right)^2 + \lambda_{1}\lambda_{4} \left(\frac{x-1}{xy-1}\right)^2 + \lambda_{2}\lambda_{3}(1-xy)^2 + \lambda_{2}\lambda_{5}y^{2} + \lambda_{3}\lambda_{4}x^{2} .
        $$
        Then the following inequality is satisfied:
        $$
            \lambda_{1}\lambda_{2} + \lambda_{1}\lambda_{3} +  \lambda_{2}\lambda_{4} + \lambda_{3}\lambda_{5} + \lambda_{4}\lambda_{5} + f(x,y) \leq 2. 
        $$
    \end{lemma}

    \begin{proof}
              Let us calculate the second derivative $\frac{\partial^{2}}{\partial x^2}\left(\frac{1-x}{1-xy}\right)^2$. 

        $$
            \frac{\partial^{2}}{\partial x^2}\left(\frac{1-x}{1-xy}\right)^2 = \frac{2(1-y)((2x-3)y + 1)}{(1-xy)^4}.
        $$
        Therefore if $(2x-3)y+1\geq 0$, then $\frac{\partial^{2}}{\partial x^2}f(x,y)\geq 0$. Thus, for $(x,y) \in [0;1)^2 \cap \{(2x-3)y + 1\geq 0\}$ we get the following estimate
        $$
            f(x,y)\leq \max \left(f\left(\max\left(0,\frac{3-\frac{1}{y}}{2}\right), y\right), f(1,y)\right).
        $$
        Our goal now is to find upper bounds for $f(0,y),f\left(\frac{3-\frac{1}{y}}{2},y\right),f(1,y).$ 

        $$
            f(0,y) = \lambda_{1}\lambda_{5} + \lambda_{1}\lambda_{4} (1-x)^2 + \lambda_{2}\lambda_{3} + \lambda_{2}\lambda_{5}\cdot 0 + \lambda_{3}\lambda_{4}x^{2} 
        $$

        $$
            \leq \lambda_{1}\lambda_{5}  + \lambda_{2}\lambda_{3} + \max(\lambda_{1}\lambda_{4},\lambda_{3}\lambda_{4}).
        $$
        Therefore we may ignore  either $\lambda_{1}\lambda_{4} + \lambda_{2}\lambda_{5}$ or $\lambda_{3}\lambda_{4} + \lambda_{2}\lambda_{5}$.

        $$
            f(1,y) = \lambda_{1}\lambda_{5} + \lambda_{1}\lambda_{4} \cdot 0 + \lambda_{2}\lambda_{3}(1-y)^2 + \lambda_{2}\lambda_{5}y^{2} + \lambda_{3}\lambda_{4}
        $$
        $$
            \leq \lambda_{1}\lambda_{5} + \lambda_{3}\lambda_{4} + \max (\lambda_{2}\lambda_{3},\lambda_{2}\lambda_{5}). 
        $$
        Therefore we may ignore  either $\lambda_{1}\lambda_{4} + \lambda_{2}\lambda_{3}$ or $\lambda_{1}\lambda_{4} + \lambda_{2}\lambda_{5}$.
        
        For $y\geq \frac{1}{3}$ we also need to estimate $f\left(\frac{3-\frac{1}{y}}{2}, y\right)$. 
        $$
         f\left(\frac{3-\frac{1}{y}}{2}, y\right) =  \frac{4}{9}\lambda_{1}\lambda_{5} + \frac{1}{9y^2}\lambda_{1}\lambda_{4}+ \lambda_{2}\lambda_{3}\left(\frac{3-3y}{2}\right)^2 + \lambda_{2}\lambda_{5}y^2 + \lambda_{3}\lambda_{4}\left(\frac{3-\frac{1}{y}}{2}\right)^2 .  
        $$
        Assume that $y\in \left[\frac{1}{3}; \frac{1}{2}\right]$. Then since
        $$
            \frac{\partial^{2}}{\partial y^2}\left(3-\frac{1}{y}\right)^2 = \frac{6 - 12y }{y^4}\geq 0,
        $$
        we have $\frac{\partial^{2}}{\partial y^2}f\left(\frac{3-\frac{1}{y}}{2}, y\right)\geq 0$ and $$
            f\left(\frac{3-\frac{1}{y}}{2}, y\right)\leq \max \left(f\left(0,\frac{1}{3}\right), f\left(\frac{1}{2}, \frac{1}{2}\right)\right).
        $$

        The estimation of $f(0,y)$ was already  done so we only need to estimate $f\left(\frac{1}{2},\frac{1}{2}\right).$

        $$
            f\left(\frac{1}{2},\frac{1}{2}\right) = \frac{4}{9}\lambda_{1}\lambda_{5} + \frac{4}{9}\lambda_{1}\lambda_{4}+ \frac{9}{16}\lambda_{2}\lambda_{3} + \frac{1}{4}\lambda_{2}\lambda_{5} + \frac{1}{4}\lambda_{3}\lambda_{4} . 
        $$
        Consider two cases: $\lambda_{1}\lambda_{5}\leq \lambda_{1}\lambda_{4}$ and $\lambda_{1}\lambda_{4}\leq \lambda_{1}\lambda_{5}$. In the first case 

        $$
             \frac{4}{9}\lambda_{1}\lambda_{5} + \frac{4}{9}\lambda_{1}\lambda_{4}\leq \lambda_{1}\lambda_{4};
        $$

        $$
            \frac{1}{4}\lambda_{3}\lambda_{4} + \frac{9}{16}\lambda_{2}\lambda_{3} \leq \max (\lambda_{3}\lambda_{4}, \lambda_{2}\lambda_{3} ),
        $$
        so we may ignore either $\lambda_{1}\lambda_{5} + \lambda_{3}\lambda_{4}$ or $\lambda_{1}\lambda_{5} + \lambda_{2}\lambda_{3}$. In the case where $\lambda_{1}\lambda_{4}\leq \lambda_{1}\lambda_{5}$ the following estimates hold.

        $$
             \frac{4}{9}\lambda_{1}\lambda_{5} + \frac{4}{9}\lambda_{1}\lambda_{4}\leq \lambda_{1}\lambda_{5};
        $$

        $$
            \frac{1}{4}\lambda_{2}\lambda_{5} + \frac{9}{16}\lambda_{2}\lambda_{3} \leq \max (\lambda_{2}\lambda_{5}, \lambda_{2}\lambda_{3} ).
        $$
        So we may ignore either $\lambda_{1}\lambda_{4} + \lambda_{2}\lambda_{5}$ or $\lambda_{1}\lambda_{4} + \lambda_{2}\lambda_{3}$. 

        As a result, we obtain that the desired estimate holds for all $$
            (x,y) \in [0;1]^{2}\cap \{ y\leq \frac{1}{2}, (2x-3)y+1\geq0\} 
        $$
        and similarly for all 
        $$
            (x,y) \in [0;1]^{2}\cap \left(\{ y\leq \frac{1}{2}, (2x-3)y+1\geq0\}\cup \{ x\leq \frac{1}{2}, (2y-3)x+1\geq0\}\right) $$
            
            $$= \left\{(a,b)\in[0;1]^{2} | a\leq \frac{1}{2} \text{ or } b\leq \frac{1}{2} \right\}.
        $$

        Let us make a change of variables:

        $$
            c  = \frac{1-x}{1-xy};
        $$

        $$
            d  = 1-xy.
        $$
        Then  
        $$
            y = \frac{1-d}{1-cd};
        $$
        $$
            \frac{1-y}{1-xy} = \frac{1-c}{1-cd};
        $$
        $$
            x = 1- cd.
        $$

        The expression in terms of new variables has the form

         $$
            \lambda_{1}\lambda_{5}\left(\frac{1-c}{1-cd}\right)^2 + \lambda_{1}\lambda_{4} c^2 + \lambda_{2}\lambda_{5}\left(\frac{1-d}{1-cd}\right)^2 + \lambda_{3}\lambda_{4}(1-cd)^{2} + \lambda_{2}\lambda_{3}d^2
        $$
        and similarly the desired inequality holds for $c\leq \frac{1}{2}$ or $d\leq \frac{1}{2}$. That is $2x - xy \geq 1$ or $xy\geq \frac{1}{2}$.  Then we also get the estimate,  if $2y-xy\geq 1$.  The remaining case is when $x\geq \frac{1}{2}, y\geq \frac{1}{2}, xy\leq \frac{1}{2}, 2x-xy\leq 1, 2y-xy\leq 1$. That is $(x,y)\in \Omega.$

        If $(x,y)\in\Omega$ then by using Lemma~\ref{lemma7.2} and Lemma~\ref{lemma3.6} we get the following chain of inequalities:

        $$
            \lambda_{1}\lambda_{2} + \lambda_{1}\lambda_{3} +  \lambda_{2}\lambda_{4} + \lambda_{3}\lambda_{5} + \lambda_{4}\lambda_{5} + f(x,y) 
        $$
        $$
            \leq   \lambda_{1}\lambda_{2} + \lambda_{1}\lambda_{3} + \lambda_{2}\lambda_{4} +\lambda_{3}\lambda_{5} + \lambda_{4}\lambda_{5} + \frac{3}{5}(\lambda_{1}\lambda_{5} + \lambda_{1}\lambda_{4} +\lambda_{2}\lambda_{3} + \lambda_{2}\lambda_{5} + \lambda_{3}\lambda_{4} ) \leq 2.
        $$  
        
    \end{proof}
    
    \begin{proposition}
    \label{bad-case}
        Let  $\{a_{ij}\}_{1\leq i<j\leq 5}$ is peculiar. Then we have
    
        $$
             \sum_{1\leq i<j\leq 5}\lambda_{i}\lambda_{j}a_{ij}^2 \leq 2.
        $$
    \end{proposition}
    
    \begin{proof}
        Since $|a_{12}|=|a_{13}|=|a_{24}|=|a_{35}|=|a_{45}|=1$ we need to estimate the remaining part, which is 

         $$
            \lambda_{1}\lambda_{5}a_{15}^2 + \lambda_{1}\lambda_{4} a_{14}^2 + \lambda_{2}\lambda_{5}\left(\frac{1-|a_{15}|}{|a_{14}|}\right)^2 + \lambda_{3}\lambda_{4}\left(\frac{1-|a_{14}|}{|a_{15}|}\right)^2 + \lambda_{2}\lambda_{3}\left(\frac{|a_{14}|+|a_{15}| - 1}{|a_{14}a_{15}|}\right)^2 .
        $$
        
    If $|a_{14}| + |a_{15}| = 1$, then we instantly lose $\lambda_{2}\lambda_{3}$ and 
       $$
            \lambda_{1}\lambda_{5}a_{15}^2 + \lambda_{1}\lambda_{4}(1-|a_{15}|)^2 \leq \max (\lambda_{1}\lambda_{5},\lambda_{1}\lambda_{4}).
       $$
    Therefore we may ignore  either $\lambda_{1}\lambda_{5} + \lambda_{2}\lambda_{3}$ or $\lambda_{1}\lambda_{4} + \lambda_{2}\lambda_{3}$. Then the estimate follows from Lemma~\ref{lemma-2}.

        If $|a_{14}| + |a_{15}| > 1$, then let us define 
        $$
            x:=\frac{1-|a_{14}|}{|a_{15}|};
        $$
        $$         y:=\frac{1-|a_{15}|}{|a_{14}|}.
        $$
        Then $x,y\in [0;1)$ and 
        $$
            |a_{15}| = \frac{1-y}{1-xy};
        $$
         $$
            |a_{14}| = \frac{1-x}{1-xy};
        $$
        $$
            \frac{|a_{14}|+|a_{15}| - 1}{|a_{14}a_{15}|} = 1-xy.
        $$

        After changing the variables, our expression becomes the following

        $$      \lambda_{1}\lambda_{5}\left(\frac{y-1}{xy-1}\right)^2 + \lambda_{1}\lambda_{4} \left(\frac{x-1}{xy-1}\right)^2 + \lambda_{2}\lambda_{3}(1-xy)^2 + \lambda_{2}\lambda_{5}y^{2} + \lambda_{3}\lambda_{4}x^{2} .
        $$ 

        Applying Lemma~\ref{lemma7.3} completes the proof.

    \end{proof}

\end{section}

\begin{section}{The case \texorpdfstring{$a_{12} = 0$}{Lg}}
\label{s8}
    \begin{lemma}
    \label{sss}
        Let $x+y+z=0$ and $(x,y,z)\in[-1;1]^3$. Then for any $a,b,c \geq 0$ it holds that 
        $$
            ax^2+by^2+cz^2 \leq a+b+c-\min(a,b,c).
        $$
        Therefore we may ignore  either $a$ or $b$ or $c$.
    \end{lemma}

    \begin{proof}
        The function $(x,y,z)\mapsto ax^2 + by^2 + cz^2$ is convex and its domain is just $$conv(\{(1,-1,0), (-1,1,0),(1,0,-1), (-1,0,1),(0,1,-1),(0,-1,1)\}).$$ 
    \end{proof}

    % \begin{lemma}
    % \label{two-zeros}
    %     Let  $\{a_{ij}\}_{1\leq i<j\leq 5}$ is admissible and $k,l,n$ be three different  numbers from $\{1,2,3,4,5\}$. Assume that   $a_{kl}=a_{kn}=0$. Then we have
    
    %     $$
    %          \sum_{1\leq i<j\leq 5}\lambda_{i}\lambda_{j}a_{ij}^2 \leq 2.
    %     $$
    % \end{lemma}

    % \begin{proof}
    %     Let $k=1,l=2,n=3.$
    %     If  $a_{12}=a_{13}=0$, then either $a_{14}=0$ or $a_{23}=0$. If $a_{14}=0$ then either $a_{15}=0$ or $a_{24}=0$. 

    %     The case $a_{12}=a_{13}=a_{14}=a_{15}=0$ is similar to the case  $\lambda_{1} =0$, which we did in Proposition~\ref{lambda-zero-case}. 
        
    %     The case $a_{12}=a_{13}=a_{14}=a_{24}=0$ is the case with missing $\lambda_{1}\lambda_{3} + \lambda_{2}\lambda_{4}$. Lemma~\ref{lemma3.3} completes the proof in this case. 

    %     The case $a_{12}=a_{13}=a_{23}=0$ is the case with missing $\lambda_{1}\lambda_{2} + \lambda_{2}\lambda_{3} + \lambda_{1}\lambda_{3}$. Lemma~\ref{lemma3.4} 
    % \end{proof}

    \begin{notation}
        We write $\pm a \pm b$ if there is a fixed choice of signs, but we do not know this choice. For example, assume that $|x|=|y|=1$ and  we have the expression $xa+yb$. Then we can rewrite it as $\pm a \pm b$.
    \end{notation}

    \begin{proposition}
        \label{zero-case}
        Let  $\{a_{ij}\}_{1\leq i<j\leq 5}$ is admissible and $a_{12}=0$. Then we have
    
        $$
             \sum_{1\leq i<j\leq 5}\lambda_{i}\lambda_{j}a_{ij}^2 \leq 2.
        $$
    \end{proposition}
    
    \begin{proof}
        Assume that the set $\{a_{ij}\}$ maximizes the value of $\sum_{1\leq i<j\leq 5}\lambda_{i}\lambda_{j}a_{ij}^2$ among all admissible sets with $a_{12}=0.$ 
    
        Let $a_{13}=0$. Then either $a_{14}=0$ or $a_{23}=0$. If $a_{14}=0$ then either $a_{15}=0$ or $a_{24}=0$. 

        The case $a_{12}=a_{13}=a_{14}=a_{15}=0$ is similar to the case $\lambda_{1} =0$, which we did in Proposition~\ref{lambda-zero-case}. 
        
        The case where $a_{12}=a_{13}=a_{14}=a_{24}=0$ is the case with missing $\lambda_{1}\lambda_{3} + \lambda_{2}\lambda_{4}$. Lemma~\ref{lemma-2} completes the proof in this case. 

        The case where $a_{12}=a_{13}=a_{23}=0$ is the case with missing $\lambda_{1}\lambda_{2} + \lambda_{2}\lambda_{3} + \lambda_{1}\lambda_{3}$. Lemma~\ref{lemma3.4} completes the proof in this case.

        Therefore we can assume that $a_{13},a_{14},a_{15},a_{23},a_{24},a_{25}$ are not zeros. Then the relations become the following.

        $$
            a_{24} = \frac{1}{a_{13}}a_{14}a_{23}; 
        $$
        $$
            a_{25} = \frac{1}{a_{13}}a_{15}a_{23};
        $$

        $$
            a_{45} = \frac{1}{a_{13}}(a_{14}a_{35} - a_{15}a_{34}).
        $$

        Consider $u=(0, a_{14},a_{15}), v=(-a_{23}, a_{34},a_{35})$.  Again, we can move $v$ inside $[-1;1]^3$ with $[u;v_{t}] \equiv [u;v]$ as in Proposition~\ref{max-form}, and hence,  conclude that $v$ must lie on the boundary $\partial [-1;1]^3.$ Similarly $(a_{42},a_{43},a_{45}), (a_{53},a_{52},a_{54}),(a_{13}, a_{34},a_{35}),(a_{41},a_{43},a_{45}),(a_{53},a_{51},a_{54}) \in \partial [-1;1]^{3}$. Thus, in each of the following six lines, at least one element is equal to 1 or -1.

        $$
            a_{32},a_{34},a_{35};
        $$
        $$
            a_{42},a_{43},a_{45};
        $$
        $$
            a_{52},a_{53},a_{54};
        $$

         $$
            a_{31},a_{34},a_{35};
        $$
        $$
            a_{41},a_{43},a_{45};
        $$
        $$
            a_{51},a_{53},a_{54}.
        $$
    Thus there are two possible  cases: $|a_{ij}|\neq 1$ for $i,j\in \{3,4,5\}$ and $|a_{34}|=1$.     

    \textbf{Case 1 ($|a_{ij}|\neq 1$ for $i,j\in \{3,4,5\}$ ).} In this case 
    $$
        |a_{13}|=|a_{14}|=|a_{15}|=|a_{23}|=|a_{24}|=|a_{25}|=1
    $$
    and $|a_{45}| = |\pm a_{35} \pm a_{34}|.$ Thus for some choice of signs, the following is true $$\pm a_{45} \pm a_{35} \pm a_{34} = 0.$$ 
    Therefore by Lemma~\ref{sss} we may ignore either $\lambda_{1}\lambda_{2} + \lambda_{3}\lambda_{4}$ or  $\lambda_{1}\lambda_{2} + \lambda_{3}\lambda_{5}$ or $\lambda_{1}\lambda_{2} + \lambda_{4}\lambda_{5}$.

    \textbf{Case 2. ($|a_{34}|=1$).} In this case there is at least one $1$ among $|a_{52}|,|a_{53}|,|a_{54}|$ and among $|a_{51}|,|a_{53}|,|a_{54}|.$ And there are only two cases: $|a_{53}|=1$ and $|a_{52}|=|a_{51}|=1.$

    \textbf{Case 2.1. ($|a_{34}|=1$, $|a_{53}|=1$).} In this case we have the following restrictions:

    $$
            |a_{24}| = \frac{1}{|a_{13}|}|a_{14}a_{23}|\leq 1; 
        $$
        $$
            |a_{25}| = \frac{1}{|a_{13}|}|a_{15}a_{23}|\leq 1;
        $$

        $$
            |a_{45}| = \frac{1}{|a_{13}|}|\pm a_{14} \pm  a_{15}|\leq 1.
        $$

    Then 
    $$
        \sum_{1\leq i < j \leq 5} \lambda_{i}\lambda_{j}a_{ij}^2 = (\lambda_{3}\lambda_{4} + \lambda_{3}\lambda_{5}) + \lambda_{1}\lambda_{5}a_{15}^2  $$
        $$     
 + \lambda_{2}\lambda_{5}\left(\frac{a_{15}a_{23}}{a_{13}}\right)^2 + \lambda_{1}\lambda_{4}a_{14}^2 + \lambda_{2}\lambda_{4} \left(\frac{a_{14}a_{23}}{a_{13}}\right)^2 + \lambda_{1}\lambda_{3}a_{13}^2 + \lambda_{2}\lambda_{3}a_{23}^2 +\lambda_{4}\lambda_{5}\left(\frac{\pm a_{14} \pm a_{15}}{a_{13}}\right)^2.   $$

    If $|a_{24}|,|a_{25}|,|a_{23}|<1 $, then we can increase $|a_{23}|.$ Therefore there are two cases: $|a_{24}|=1$ and $|a_{23}|=1$.

     \textbf{Case 2.1.1. ($|a_{34}|=1$, $|a_{53}|=1$, $|a_{24}|=1$).} The non-trivial part of expression is 

            $$     
    \lambda_{1}\lambda_{5}a_{15}^2 + \lambda_{2}\lambda_{5}\left(\frac{a_{15}}{a_{14}}\right)^2 + \lambda_{1}\lambda_{4}a_{14}^2 + \lambda_{1}\lambda_{3}a_{13}^2 + \lambda_{2}\lambda_{3}\left(\frac{a_{13}}{a_{14}}\right)^2 +\lambda_{4}\lambda_{5}\left(\frac{\pm a_{14} \pm a_{15}}{a_{13}}\right)^2   $$
    and the restrictions are
    
    $$
        |a_{15}|\leq |a_{14}|;
    $$
    $$
        |a_{13}|\leq |a_{14}|;
    $$
    $$
        |\pm a_{14} \pm a_{15}| \leq |a_{13}|.
    $$
    If $|a_{14}|<1$ and $|\pm a_{14} \pm a_{15}| < |a_{13}|$, then it is possible to increase both $|a_{14}|$ and $|a_{15}|$. Therefore there are two cases: $|a_{14}|=1$ and $|\pm a_{14} \pm a_{15}| = |a_{13}|$.

    \textbf{Case 2.1.1.1. ($|a_{34}|=1$, $|a_{53}|=1$, $|a_{24}|=1$, $|a_{14}|=1$).} The non-trivial part of expression is 

                $$     
        (\lambda_{1}\lambda_{5} + \lambda_{2}\lambda_{5})a_{15}^2 + (\lambda_{1}\lambda_{3} + \lambda_{2}\lambda_{3})a_{13}^2+\lambda_{4}\lambda_{5}\left(\frac{\pm 1 \pm a_{15}}{a_{13}}\right)^2   $$
        and the restriction is
        $$
            |\pm 1 \pm a_{15}| \leq |a_{13}|.
        $$
        Then $|a_{13}|\in \left[|\pm 1 \pm a_{15}|; 1\right]$. So it suffices to substitute only $|a_{13}|=|\pm 1 \pm a_{15}|$ and $|a_{13}|=1.$

        \textbf{Case 2.1.1.1.1. ($|a_{34}|=1$, $|a_{53}|=1$, $|a_{24}|=1$, $|a_{14}|=1$, $|a_{13}|=|\pm 1 \pm a_{15}|$).} The non-trivial part of expression is 

                $$     
        (\lambda_{1}\lambda_{5} + \lambda_{2}\lambda_{5})a_{15}^2 + (\lambda_{1}\lambda_{3} + \lambda_{2}\lambda_{3})(1-|a_{15}|)^2 \leq \max (\lambda_{1}\lambda_{5} + \lambda_{2}\lambda_{5}, \lambda_{1}\lambda_{3} + \lambda_{2}\lambda_{3}).  $$
        Therefore we may ignore  either $\lambda_{1}\lambda_{2} +\lambda_{1}\lambda_{5} + \lambda_{2}\lambda_{5} $ or $\lambda_{1}\lambda_{2} + \lambda_{1}\lambda_{3} + \lambda_{2}\lambda_{3}$.
        
        \textbf{Case 2.1.1.1.2. ($|a_{34}|=1$, $|a_{53}|=1$, $|a_{24}|=1$, $|a_{14}|=1$, $|a_{13}|=1$).} The non-trivial part of expression is 

                $$     
        (\lambda_{1}\lambda_{5} + \lambda_{2}\lambda_{5})a_{15}^2 + \lambda_{4}\lambda_{5}(1-|a_{15}|)^2 \leq \max (\lambda_{1}\lambda_{5} + \lambda_{2}\lambda_{5}, \lambda_{4}\lambda_{5}).
        $$
    Therefore we may ignore either $\lambda_{1}\lambda_{2} +\lambda_{1}\lambda_{5} + \lambda_{2}\lambda_{5} $ or $\lambda_{1}\lambda_{2} + \lambda_{4}\lambda_{5}$.

       \textbf{Case 2.1.1.2. ($|a_{34}|=1$, $|a_{53}|=1$, $|a_{24}|=1$, $|a_{13}|=|\pm a_{14} \pm a_{15}|$).} The non-trivial part of expression is 

            $$     
    \lambda_{1}\lambda_{5}a_{15}^2 + \lambda_{2}\lambda_{5}\left(\frac{a_{15}}{a_{14}}\right)^2 + \lambda_{1}\lambda_{4}a_{14}^2 + \lambda_{1}\lambda_{3}(\pm a_{14} \pm a_{15})^2 + \lambda_{2}\lambda_{3}\left(\frac{\pm a_{14}\pm a_{15}}{a_{14}}\right)^2 $$

    $$     
    = (\lambda_{1}\lambda_{5}a_{14}^2 + \lambda_{2}\lambda_{5})\left(\frac{a_{15}}{a_{14}}\right)^2 + \lambda_{1}\lambda_{4}a_{14}^2 + (\lambda_{1}\lambda_{3}a_{14}^2 + \lambda_{2}\lambda_{3})\left(1 - \left |\frac{ a_{15}}{a_{14}}\right|\right)^2 $$
    $$
        \leq \lambda_{1}\lambda_{4} + \max (\lambda_{1}\lambda_{5} + \lambda_{2}\lambda_{5}, \lambda_{1}\lambda_{3}+
        \lambda_{2}\lambda_{3}).
    $$
    Therefore we may ignore either $\lambda_{1}\lambda_{2} +\lambda_{1}\lambda_{5} + \lambda_{2}\lambda_{5} $ or $\lambda_{1}\lambda_{2} + \lambda_{1}\lambda_{3} + \lambda_{2}\lambda_{3}$.

     \textbf{Case 2.1.2. ($|a_{34}|=1$, $|a_{53}|=1$, $|a_{23}|=1$).} The non-trivial part of the expression is 

        $$     
\lambda_{1}\lambda_{5}a_{15}^2  + \lambda_{2}\lambda_{5}\left(\frac{a_{15}}{a_{13}}\right)^2 + \lambda_{1}\lambda_{4}a_{14}^2 + \lambda_{2}\lambda_{4} \left(\frac{a_{14}}{a_{13}}\right)^2 + \lambda_{1}\lambda_{3}a_{13}^2  +\lambda_{4}\lambda_{5}\left(\frac{\pm a_{14} \pm a_{15}}{a_{13}}\right)^2   $$

$$
    = (\lambda_{1}\lambda_{5}a_{13}^2  + \lambda_{2}\lambda_{5})\left(\frac{a_{15}}{a_{13}}\right)^2 + (\lambda_{1}\lambda_{4}a_{13}^2 + \lambda_{2}\lambda_{4}) \left(\frac{a_{14}}{a_{13}}\right)^2  +\lambda_{4}\lambda_{5}\left(\frac{\pm a_{14} \pm a_{15}}{a_{13}}\right)^2 + \lambda_{1}\lambda_{3}a_{13}^2. 
$$
Then by Lemma~\ref{sss} we may ignore  either $\lambda_{1}\lambda_{5} + \lambda_{2}\lambda_{5} + \lambda_{1}\lambda_{2}$ or $\lambda_{1}\lambda_{4} + \lambda_{2}\lambda_{4} + \lambda_{1}\lambda_{2}$ or  $\lambda_{1}\lambda_{2} + \lambda_{4}\lambda_{5}$.

    \textbf{Case 2.2. ($|a_{34}|=1$, $|a_{51}|=|a_{52}|=1$).} In this case we have the following restrictions:

    $$
            |a_{14}|=|a_{24}|; 
        $$
        $$
            |a_{13}|=|a_{23}|;
        $$

        $$
            |a_{45}| = \frac{1}{|a_{13}|}(1 - |a_{14}a_{35}|).
        $$
    Then 
    $$
        \sum_{1\leq i < j \leq 5} \lambda_{i}\lambda_{j}a_{ij}^2 = (\lambda_{3}\lambda_{4} + \lambda_{1}\lambda_{5} + \lambda_{2}\lambda_{5})  $$
        $$     
 + (\lambda_{1}\lambda_{4} + \lambda_{2}\lambda_{4})a_{14}^2 + (\lambda_{1}\lambda_{3} + \lambda_{2}\lambda_{3})a_{13}^2 + \lambda_{3}\lambda_{5}a_{35}^2 +\lambda_{4}\lambda_{5}\frac{1}{a_{13}^2}(1-|a_{14}a_{35}|)^2.   $$

    We have $|a_{45}|\leq 1$, thus 

    $$
        \frac{1-|a_{13}|}{|a_{14}|}\leq |a_{35} | \leq 1
    $$
    and since this function is convex with respect to $|a_{35}|$ it suffices to substitute $|a_{35}| = \frac{1-|a_{13}|}{|a_{14}|}$ and $|a_{35}|=1$.

    \textbf{Case 2.2.1. ($|a_{34}|=1$, $|a_{51}|=|a_{52}|=1$, $|a_{35}| = \frac{1-|a_{13}|}{|a_{14}|}$).} The non-trivial part of the expression is 
     $$      (\lambda_{1}\lambda_{4} + \lambda_{2}\lambda_{4})a_{14}^2 + (\lambda_{1}\lambda_{3} + \lambda_{2}\lambda_{3})a_{13}^2 + \lambda_{3}\lambda_{5}\left(\frac{1-|a_{13}|}{|a_{14}|}\right)^2   $$
     and the remaining restrictions are $|a_{13}|\leq 1, |a_{14}|\leq 1, |a_{13}| + |a_{14}|\geq 1.$ Therefore $$ 1-|a_{13}|\leq |a_{14}| \leq 1$$
     and since the function is convex with respect to $|a_{14}|$, it suffices to substitute $|a_{14}| = 1 - |a_{13}|$ and $|a_{14}|=1$. 

     \textbf{Case 2.2.1.1. ($|a_{34}|=1$, $|a_{51}|=|a_{52}|=1$, $|a_{35}| = \frac{1-|a_{13}|}{|a_{14}|}$, $|a_{14}| = 1 - |a_{13}|$).} The non-trivial part of the expression is 
     $$      (\lambda_{1}\lambda_{4} + \lambda_{2}\lambda_{4})(1-|a_{13}|)^2 + (\lambda_{1}\lambda_{3} + \lambda_{2}\lambda_{3})a_{13}^2 \leq \max (\lambda_{1}\lambda_{4} + \lambda_{2}\lambda_{4}, \lambda_{1}\lambda_{3} + \lambda_{2}\lambda_{3}).   $$ 
     Therefore we may ignore  either $\lambda_{1}\lambda_{2} + \lambda_{1}\lambda_{3} + \lambda_{2}\lambda_{3}$ or $\lambda_{1}\lambda_{2} + \lambda_{1}\lambda_{4} + \lambda_{2}\lambda_{4}$ and Lemma~\ref{lemma3.4} completes the proof of this case. 

     \textbf{Case 2.2.1.2. ($|a_{34}|=1$, $|a_{51}|=|a_{52}|=1$, $|a_{35}| = \frac{1-|a_{13}|}{|a_{14}|}$, $|a_{14}| = 1$).} The non-trivial part of the expression is 
     $$       (\lambda_{1}\lambda_{3} + \lambda_{2}\lambda_{3})a_{13}^2 + \lambda_{3}\lambda_{5}(1-|a_{13}|)^2 \leq \max (\lambda_{1}\lambda_{3} + \lambda_{2}\lambda_{3}, \lambda_{3}\lambda_{5}).   $$
    Therefore we may ignore  either $\lambda_{1}\lambda_{2} + \lambda_{1}\lambda_{3} + \lambda_{2}\lambda_{3} $ or $\lambda_{1}\lambda_{2} + \lambda_{3}\lambda_{5}$. Lemmas ~\ref{lemma-2} and \ref{lemma3.4} complete the proof of this case.
    
    \textbf{Case 2.2.2. ($|a_{34}|=1$, $|a_{51}|=|a_{52}|=1$, $|a_{35}| = 1$).} The non-trivial part of the expression is 

    $$     
 (\lambda_{1}\lambda_{4} + \lambda_{2}\lambda_{4})a_{14}^2 + (\lambda_{1}\lambda_{3} + \lambda_{2}\lambda_{3})a_{13}^2  +\lambda_{4}\lambda_{5}\frac{1}{a_{13}^2}(1-|a_{14}|)^2.   $$

 This case is similar to the Case $2.2.1$ and we may ignore either $\lambda_{1}\lambda_{2} + \lambda_{1}\lambda_{3} + \lambda_{2}\lambda_{3}$ or  $\lambda_{1}\lambda_{2} + \lambda_{1}\lambda_{4} + \lambda_{2}\lambda_{4}$ or $\lambda_{1}\lambda_{2} + \lambda_{4}\lambda_{5}.$
    \end{proof}
    
\end{section}

\section*{Acknowledgement}
\noindent We would like to thank Nikita Kalinin for the help with the preparation of this article.

\textit{E-mail address:} \href{mailto:arkadiy.aliev@gmail.com}{arkadiy.aliev@gmail.com} 

\end{document}